\documentclass{amvart}
\usepackage{amssymb}
\usepackage{url}

\newtheorem{theorem}{Theorem}[section]
\newtheorem{lemma}[theorem]{Lemma}

\theoremstyle{definition}

\theoremstyle{remark}

\numberwithin{equation}{section}

\begin{document}

\title[Mathieu series]{A solution to an open problem on the Mathieu series\\ posed by Hoorfar and Qi}
\author{Gerg\H{o} Nemes}
\address{Lor\'and E\"otv\"os University\\ H-1117 Budapest,
P\'azm\'any P\'eter s\'et\'any 1/C, Hungary}
\email{nemesgery@gmail.com}
\dedicatory{}
\begin{abstract}
In this paper, the author gives a solution to an open problem about the famous Mathieu series, that is, he obtains a sharp double inequality for bounding this series.
\end{abstract}
\maketitle

\section{Introduction}\label{sec1}

We consider the series
\begin{equation}\label{eq0}
S\left( r \right) = \sum\limits_{n \ge 1} {\frac{{2n}}{{\left( {n^2  + r^2 } \right)^2 }}} ,\;\;r > 0
\end{equation}
called Mathieu's series in the literature. It was introduced by Mathieu in his 1890 work on elasticity of solid bodies. Several interesting problems, solutions and bounds for the Mathieu series can be found in the works we collected in the references. Alzer, Brenner and Ruehr \cite{alzer} showed that the best constants $\alpha$ and $\beta$ for which
\[
\frac{1}{{r^2  + \alpha }} < S\left( r \right) < \frac{1}{{r^2  + \beta }}
\]
holds, are $\alpha = \frac{1}{2\zeta \left(3\right)}$ and $\beta = \frac{1}{6}$, where $\zeta$ denotes Riemann's Zeta function. Our main purpose is to prove the following problem proposed by A. Hoorfar and F. Qi \cite{qi}:
find the best possible constants $a$ and $b$ such that
\begin{equation}\label{eq9}
\cfrac{1}{{r^2  + \cfrac{1}{2} - \cfrac{{4r^2  + 1}}{{12}}\left( {r^2  + a} \right)^{ - 1} }} < S\left( r \right) < \cfrac{1}{{r^2  + \cfrac{1}{2} - \cfrac{{4r^2  + 1}}{{12}}\left( {r^2  + b} \right)^{ - 1} }}
\end{equation}
holds true for all $r > 0$. They showed that $a \leq 3/2$ and $b \geq 1/4$. The following theorem gives the complete answer.

\begin{theorem}\label{th1} For any real number $r > 0$, the inequality \eqref{eq9} holds with the best possible constants $a = \cfrac{{\zeta \left( 3 \right)}}{{6\zeta \left( 3 \right) - 6}} = 0.9915168156\ldots$ and $b = \cfrac{{13}}{{30}}$.
\end{theorem}

\section{Lemmas}

To prove our theorem we define a function $\alpha$ for any real number $r > 0$ by the equality
\[
S\left( r \right) = \cfrac{1}{{r^2  + \cfrac{1}{2} - \cfrac{{4r^2  + 1}}{{12}}\left( {r^2  + \alpha \left( r \right)} \right)^{-1}}}.
\]
Theorem~\ref{th1} is the corollary of the following lemmas.

\begin{lemma}\label{le1}For any real number $r > 2.57$, $\alpha(r)$ is a strictly decreasing function of $r$, satisfies the inequalities $\frac{{13}}{{30}} < \alpha \left( r \right) < \frac{{\zeta \left( 3 \right)}}{{6\zeta \left( 3 \right) - 6}}$, and $\lim _{r \to +\infty } \alpha \left( r \right) = \frac{{13}}{{30}}$.
\end{lemma}

\begin{lemma}\label{le2}For any real number $0< r \leq 0.3$, $\alpha(r)$ is a strictly decreasing function of $r$ and satisfies the inequalities $\frac{{13}}{{30}} < \alpha \left( r \right) < \frac{{\zeta \left( 3 \right)}}{{6\zeta \left( 3 \right) - 6}}$.
\end{lemma}

\begin{lemma}\label{le3}For any real number $0.3 < r \leq 2.57$, $\alpha(r)$ satisfies the inequalities $\frac{{13}}{{30}} < \alpha \left( r \right) < \frac{{\zeta \left( 3 \right)}}{{6\zeta \left( 3 \right) - 6}}$.
\end{lemma}

In this paper we use the following definition of the Bernoulli numbers:
\[
\frac{x}{e^{x}-1} = 1 - \frac{1}{2}x + B_2 \frac{{x^2 }}{{2!}} - B_4 \frac{{x^4 }}{{4!}} + B_6 \frac{{x^6 }}{{6!}} -  \cdots ,\;\;\left| x \right| < 2\pi ,\;\; x \neq 0,
\]
thus the first few are
\[
B_2  = \frac{1}{6},\;\;B_4  = \frac{1}{{30}},\;\;B_6  = \frac{1}{{42}},\;\;B_8  = \frac{1}{{30}},\;\;B_{10}  = \frac{5}{{66}},\;\; \ldots \,.
\]
To prove our lemmas we shall use the following theorem of Russell \cite{russel}.
\begin{lemma}\label{th2} Let $S\left(r\right)$ be defined by ~(\ref{eq0}) and $f\left( x \right)= x/\left( {e^x  - 1} \right)$ if $x \neq 0$, $f\left(0\right) = 1$. Then
\[
S\left( r \right) = \frac{1}{{r^2 }} - \sum\limits_{i = 1}^k {\frac{{B_{2i} }}{{r^{2i + 2} }}}  + R_k \left( r \right) \quad \left( {r > 0,\; k = 1,2, \ldots } \right),
\]
where $R_k \left( r \right)$ may be expressed in either of the forms
\begin{equation*}
\begin{split}
R_k \left( r \right) & = \frac{{\left( { - 1} \right)^k }}{{r^{2k + 2} }}\int_0^\infty  {f^{\left( {2k + 1} \right)} \left( x \right)\cos \left( {rx} \right)dx}\\ 
& = \frac{{\left( { - 1} \right)^{k + 1} }}{{r^{2k + 3} }}\int_0^\infty  {f^{\left( {2k + 2} \right)} \left( x \right)\sin \left( {rx} \right)dx}
\end{split}
\end{equation*}
and
\[
\left| {R_k \left( r \right)} \right| < \frac{1}{{r^{2k + 2} }}\frac{\pi }{2}\binom{k + 1/2}{k}B_{2k} .
\]
\end{lemma}

\section{Proofs of the lemmas}\label{sec2}

\begin{proof}[Proof of Lemma~\ref{le1}.]
From the definition of $\alpha \left(r\right)$ we have
\begin{equation}\label{eq1}
\alpha \left( r \right) = \frac{{\left( {12r^4  + 2r^2  - 1} \right)S\left( r \right) - 12r^2 }}{{12 - \left( {12r^2  + 6} \right)S\left( r \right)}}.
\end{equation}
According to Hoorfar and Qi we find that
\[
\frac{1}{4} \le \alpha \left( r \right) \le \frac{3}{2},
\]
if $r>0$. We have to improve this for $r>2.57$. Note that if
\begin{equation}\label{eq2}
A\left( r \right) < S\left( r \right) < B\left( r \right)
\end{equation}
and
\begin{equation}\label{eq3}
\frac{1}{{r^2  + 1/2}} < A\left( r \right),\;B\left( r \right),
\end{equation}
then
\begin{equation}\label{eq4}
\frac{{\left( {12r^4  + 2r^2  - 1} \right)B\left( r \right) - 12r^2 }}{{12 - \left( {12r^2  + 6} \right)B\left( r \right)}} < \alpha \left( r \right) < \frac{{\left( {12r^4  + 2r^2  - 1} \right)A\left( r \right) - 12r^2 }}{{12 - \left( {12r^2  + 6} \right)A\left( r \right)}}.
\end{equation}
From Lemma~\ref{th2} we have
\[
S\left( r \right) > \frac{1}{{r^2 }} - \frac{1}{{6r^4 }} - \frac{1}{{30r^6 }} - \frac{1}{{42r^8 }} - \left| {\frac{\pi }{2}\binom{7/2 }{ 3}\frac{1}{{42}}} \right|\frac{1}{{r^8 }} > \frac{1}{{r^2 }} - \frac{1}{{6r^4 }} - \frac{1}{{30r^6 }} - \frac{{53}}{{500r^8 }} > \frac{1}{{r^2  + 1/2}},
\]
if $r>1$. Let
\[
G \left(r\right)= \frac{1}{{r^2 }} - \frac{1}{{6r^4 }} - \frac{1}{{30r^6 }} - \frac{{53}}{{500r^8 }},
\]
then
\begin{equation}\label{eq5}
\alpha \left( r \right) < \frac{{\left( {12r^4  + 2r^2  - 1} \right)G\left( r \right) - 12r^2 }}{{12 - \left( {12r^2  + 6} \right)G\left( r \right)}} = \frac{{2600r^6  + 1758r^4  + 268r^2  - 159}}{{6000r^6  - 2100r^4  - 2208r^2  - 954}} < \frac{{\zeta \left( 3 \right)}}{{6\zeta \left( 3 \right) - 6}},
\end{equation}
if $r>1.3$. Applying Russell's theorem again, we obtain
\[
\frac{1}{{r^2  + 1/2}} < S\left( r \right) < \frac{1}{{r^2 }} - \frac{1}{{6r^4 }} - \frac{1}{{30r^6 }} - \frac{1}{{42r^8 }} + \left| {\frac{\pi }{2}\binom{7/2 }{ 3}\frac{1}{{42}}} \right|\frac{1}{{r^8 }} < \frac{1}{{r^2 }} - \frac{1}{{6r^4 }} - \frac{1}{{30r^6 }} + \frac{{3}}{{50r^8 }},
\]
if $r>0$. Let
\[
K\left( r \right)= \frac{1}{{r^2 }} - \frac{1}{{6r^4 }} - \frac{1}{{30r^6 }} + \frac{3}{{50r^8 }},
\]
then
\[
\alpha \left( r \right) > \frac{{\left( {12r^4  + 2r^2  - 1} \right)K\left( r \right) - 12r^2 }}{{12 - \left( {12r^2  + 6} \right)K\left( r \right)}} = \frac{{260r^6  - 123r^4  - 23r^2  + 9}}{{600r^6  - 210r^4  + 78r^2  + 54}}.
\]
From this and from (\ref{eq5}) we find
\[
\mathop {\lim }\limits_{r \to +\infty } \frac{{260r^6  - 123r^4  - 23r^2  + 9}}{{600r^6  - 210r^4  + 78r^2  + 54}} \leq \mathop {\lim }\limits_{r \to +\infty } \alpha \left( r \right) \leq \mathop {\lim }\limits_{r \to +\infty } \frac{{2600r^6  + 1758r^4  + 268r^2  - 159}}{{6000r^6  - 2100r^4  - 2208r^2  - 954}}.
\]
Since
\[
\mathop {\lim }\limits_{r \to +\infty } \frac{{260r^6  - 123r^4  - 23r^2  + 9}}{{600r^6  - 210r^4  + 78r^2  + 54}} = \mathop {\lim }\limits_{r \to +\infty } \frac{{2600r^6  + 1758r^4  + 268r^2  - 159}}{{6000r^6  - 2100r^4  - 2208r^2  - 954}} = \frac{{13}}{{30}}
\]
we obtain
\[
\mathop {\lim }\limits_{r \to +\infty } \alpha \left( r \right) = \frac{{13}}{{30}}.
\]

Now we consider the monotonicity of $\alpha$. We will show that the value of the derivative $\alpha'\left(r\right)$ is negative for any real number $r > 2.57$. A straightforward calculation gives
\begin{multline}
 - \frac{3}{{2r}}\left( {\left( {r^2  + \frac{1}{2}} \right)S\left( r \right) - 1} \right)^2 \alpha '\left( r \right) = \left( {3r^4 + 3r^2  + \frac{1}{2}} \right)S^2 \left( r \right) - \left( {6r^2  + 2} \right)S\left( r \right)\\
 - \left( {2r^2  + \frac{1}{2}} \right)\sum\limits_{n \ge 1} {\frac{{2n}}{{\left( {r^2  + n^2 } \right)^3 }}}  + 3.
\end{multline}
Thus it is sufficient to show that for $r>2.57$,
\begin{equation}\label{eq6}
T\left(r\right) = \left( {3r^4  + 3r^2  + \frac{1}{2}} \right)S^2 \left( r \right) - \left( {6r^2  + 2} \right)S\left( r \right) - \left( {2r^2  + \frac{1}{2}} \right)\sum\limits_{n \ge 1} {\frac{{2n}}{{\left( {r^2  + n^2 } \right)^3 }}}  + 3 > 0.
\end{equation}
From Lemma~\ref{th2} we have
\[
S\left( r \right) = \frac{1}{{r^2 }} - \frac{1}{{6r^4 }} - \frac{1}{{30r^6 }} - \frac{1}{{42r^8 }} - \frac{1}{{30r^{10} }} + R_4 \left( r \right),
\]
\begin{equation*}
\begin{split}
S^2 \left( r \right) = \frac{{2R_4 \left( r \right)}}{{r^2 }} - \frac{{R_4 \left( r \right) - 3}}{{3r^4 }} & - \frac{{R_4 \left( r \right) + 5}}{{15r^6 }} - \frac{{60R_4 \left( r \right) + 49}}{{1260r^8 }}\\
 & - \frac{{42R_4 \left( r \right) + 23}}{{630r^{10} }} + M\left( r \right),
\end{split}
\end{equation*}
\[
\sum\limits_{n \ge 1} {\frac{{2n}}{{\left( {r^2  + n^2 } \right)^3 }}}  =  - \frac{1}{{4r}}S'\left( r \right) = \frac{1}{{2r^4 }} - \frac{1}{{6r^6 }} - \frac{1}{{20r^8 }} - \frac{1}{{21r^{10} }} - \frac{1}{{12r^{12} }} - \frac{1}{{4r}}R_4 ^\prime  \left( r \right),
\]
where
\begin{equation} \label{eq7}
M\left( r \right) =  - \frac{{121}}{{2100r^{12} }} + \frac{4}{{315r^{14} }} + \frac{{41}}{{14700r^{16} }} + \frac{1}{{630r^{18} }} + \frac{1}{{900r^{20} }} + R_4^2 \left( r \right).
\end{equation}
By substituting these into (\ref{eq6}) we obtain
\begin{equation*}
\begin{split}
T\left(r\right) = 3r^4 M\left( r \right) & + 3r^2 M\left( r \right) + \frac{{R_4 ^\prime  \left( r \right)r}}{2} + \frac{{R_4 ^\prime  \left( r \right)}}{{8r}} - \frac{{R_4 \left( r \right)}}{{5r^2 }} - \frac{{107R_4 \left( r \right)}}{{210r^4 }} - \frac{{79R_4 \left( r \right)}}{{210r^6 }}\\
 & - \frac{{282R_4 \left( r \right) - 301}}{{1260r^8 }} - \frac{{6R_4 \left( r \right) - 43}}{{180r^{10} }} + \frac{1}{{24r^{12} }} + \frac{{6R_4 \left( r \right) + M\left( r \right)}}{2}.
\end{split}
\end{equation*}
From Lemma~\ref{th2} we have
\begin{equation*}
\begin{split}
R_4 \left( r \right) =  - \frac{5}{{66r^{12} }} - \frac{{691}}{{2730r^{14} }} + R_6 \left( r \right) & <  - \frac{5}{{66r^{12} }} - \frac{{691}}{{2730r^{14} }} + \frac{1}{{r^{14} }}\frac{\pi }{2}\binom{13/2 }{ 6}B_{12} \\
 & < - \frac{5}{{66r^{12} }} + \frac{{1}}{{r^{14} }}.
\end{split}
\end{equation*}
Thus if $r \geq 1$,
\begin{equation*}
\begin{split}
 &  - \frac{{R_4 \left( r \right)}}{{5r^2 }} - \frac{{107R_4 \left( r \right)}}{{210r^4 }} - \frac{{79R_4 \left( r \right)}}{{210r^6 }} - \frac{{282R_4 \left( r \right) - 301}}{{1260r^8 }} - \frac{{6R_4 \left( r \right) - 43}}{{180r^{10} }}\\
 &  > \frac{{43}}{{180r^8 }} + \frac{{43}}{{180r^{10} }} + \frac{1}{{66r^{14} }} - \frac{{2237}}{{13860r^{16} }} - \frac{{6667}}{{13860r^{18} }} - \frac{{4979}}{{13860r^{20} }} - \frac{{3067}}{{13860r^{22} }} - \frac{1}{{30r^{24} }}\\
 &   \ge \frac{{43}}{{180r^8 }} + \frac{{43}}{{180r^{10} }} + \frac{1}{{66r^{14} }} - \frac{{1451}}{{1155r^{16} }}.
\end{split}
\end{equation*}
This means that
\begin{equation*}
\begin{split}
T\left(r\right) > 3r^4 M\left( r \right) + 3r^2 M\left( r \right) & + \frac{{R'_4 \left( r \right)r}}{2} + \frac{{R'_4 \left( r \right)}}{{8r}} + \frac{{43}}{{180r^8 }} + \frac{{43}}{{180r^{10} }}\\
& + \frac{1}{{24r^{12} }} + \frac{1}{{66r^{14} }} - \frac{{1451}}{{1155r^{16} }} + \frac{{6R_4 \left( r \right) + M\left( r \right)}}{2}.
\end{split}
\end{equation*}
From (\ref{eq7}) it is clear that
\[
M\left( r \right) >  - \frac{{121}}{{2100r^{12} }} + \frac{4}{{315r^{14} }} + \frac{{41}}{{14700r^{16} }}.
\]
We also have
\begin{equation*}
\begin{split}
R_4 \left( r \right) =  - \frac{5}{{66r^{12} }} - \frac{{691}}{{2730r^{14} }} + R_6 \left( r \right) & >  - \frac{5}{{66r^{12} }} - \frac{{691}}{{2730r^{14} }} - \frac{1}{{r^{14} }}\frac{\pi }{2}\binom{13/2 }{ 6}B_{12} \\
 & >  - \frac{5}{{66r^{12} }} - \frac{{3967}}{{2730r^{14} }},
\end{split}
\end{equation*}
and thus 
\begin{equation}\label{eq8}
T\left(r\right) > \frac{{R'_4 \left( r \right)r}}{2} + \frac{{R'_4 \left( r \right)}}{{8r}} + \frac{{104}}{{1575r^8 }} + \frac{{164}}{{1575r^{10} }} - \frac{{13579}}{{80850r^{12} }} - \frac{{27302953}}{{6306300r^{14} }} - \frac{{405829}}{{323400r^{16} }}.
\end{equation}
Now we are going to estimate ${R_4 ^\prime  \left( r \right)}$. We have
\[
R'_4 \left( r \right) = \left( - \frac{5}{{66r^{12} }} - \frac{{691}}{{2730r^{14} }} + R_6 \left( r \right)\right)' = \frac{{10}}{{11r^{13} }} + \frac{{691}}{{195r^{15} }} + R'_6 \left( r \right),
\]
where
\begin{equation*}
\begin{split}
R'_6 \left( r \right) & = \left( {\frac{1}{{r^{14} }}\int_0^\infty  {f^{\left( {13} \right)} \left( x \right)\cos \left( {rx} \right)dx} } \right)^\prime \\ 
 & =  - \frac{{14}}{{r^{15} }}\int_0^\infty  {f^{\left( {13} \right)} \left( x \right)\cos \left( {rx} \right)dx}  - \frac{1}{{r^{14} }}\int_0^\infty  {xf^{\left( {13} \right)} \left( x \right)\sin \left( {rx} \right)dx} \\
 & =  - \frac{{14}}{r}R_6 \left( r \right) - \frac{1}{{r^{14} }}\int_0^\infty  {xf^{\left( {13} \right)} \left( x \right)\sin \left( {rx} \right)dx} .
\end{split}
\end{equation*}
Russell showed that
\begin{equation*}
\left| {f^{\left( k \right)} \left( x \right)} \right| \le \frac{{k!2\cdot\zeta \left( k \right)}}{{\left( {2\pi } \right)^k }},
\end{equation*}
\begin{equation*}
\left| {f^{\left( k \right)} \left( {2\pi x} \right)} \right| \le \frac{{k!}}{{\left( {2\pi } \right)^k }}\sum\limits_{n \ge 1} {\frac{{2n}}{{\left( {n^2  + x^2 } \right)^{\left( {k + 1} \right)/2} }}} ,
\end{equation*}
if $k \geq 2$. Diananda \cite{diananda} proved the inequality
\begin{equation*}
\sum\limits_{n \ge 1} {\frac{{2n}}{{\left( {n^2  + x^2 } \right)^\mu  }}}  < \frac{1}{{\left( {\mu  - 1} \right)x^{2\mu  - 2} }},\;\;\mu  > 1,\;\;x > 0.
\end{equation*}
Thus
\[
\left| {f^{\left( k \right)} \left( {2\pi x} \right)} \right| < \frac{{k!}}{{\left( {2\pi } \right)^k }}\frac{2}{{\left( {k - 1} \right)x^{k - 1} }},\;\;k > 1,\;\;x > 0.
\]
From these we obtain
\begin{equation*}
\begin{split}
\left| {R'_6 \left( r \right)} \right| & \le \frac{{14}}{r}\left| {R_6 \left( r \right)} \right| + \frac{1}{{r^{14} }}\int_0^{2\pi} {x\left| {f^{\left( {13} \right)} \left( x \right)} \right|\left| {\sin \left( {rx} \right)} \right|dx} \\
  & + \frac{4\pi^2}{{r^{14} }}\int_1^\infty  {x\left| {f^{\left( {13} \right)} \left( 2\pi x \right)} \right|\left| {\sin \left( {2 \pi rx} \right)} \right|dx}\\
  & < \frac{{14}}{r}\frac{1}{{r^{14} }}\frac{\pi }{2}\binom{13/2 }{ 6}B_{12}  + \frac{1}{{r^{14} }}\int_0^{2\pi} {x\frac{{13! \cdot 2 \cdot \zeta \left( {13} \right)}}{{\left( {2\pi } \right)^{13} }}dx}  + \frac{4\pi^2}{{r^{14} }}\int_1^\infty  {x\frac{{13!}}{{\left( {2\pi } \right)^{13} }}} \frac{2}{{12x^{12} }}dx \\
  & <  \frac{{{\rm 1051}}}{{100r^{14} }} + \frac{{49}}{{3r^{15} }}.
\end{split}
\end{equation*}
Hence we finally have
\[
R'_4 \left( r \right) > \frac{{10}}{{11r^{13} }} + \frac{{691}}{{195r^{15} }} - \frac{{1051}}{{100r^{14} }} - \frac{{49}}{{3r^{15} }} = \frac{{10}}{{11r^{13} }} - \frac{{1051}}{{100r^{14} }} - \frac{{2494}}{{195r^{15} }}.
\]
Plugging this into the right-hand side of (\ref{eq8}), we find that
\[
T\left( r \right) > \frac{{104}}{{1575r^8 }} + \frac{{164}}{{1575r^{10} }} + \frac{{23171}}{{80850r^{12} }} - \frac{{1051}}{{200r^{13} }} - \frac{{16728577}}{{1576575r^{14} }} - \frac{{1051}}{{800r^{15} }} - \frac{{11997107}}{{4204200r^{16} }}.
\]
Let
\[
P\left( r \right)= \frac{{104}}{{1575}}r^8  + \frac{{164}}{{1575}}r^6  + \frac{{23171}}{{80850}}r^4  - \frac{{1051}}{{200}}r^3  - \frac{{16728577}}{{1576575}}r^2  - \frac{{1051}}{{800}}r - \frac{{11997107}}{{4204200}},
\]
then
\begin{equation*}
\begin{split}
P\left( {r + 2.57} \right) =& \frac{{104}}{{1575}}r^8  + \frac{{53456}}{{39375}}r^7  + \frac{{12123418}}{{984375}}r^6  + \frac{{1584201713}}{{24609375}}r^5\\
& + \frac{{32175150922307}}{{151593750000}}r^4  + \frac{{8482041695506649}}{{18949218750000}}r^3  + \frac{{62685536289049749}}{{111718750000000}}r^2\\
& + \frac{{263340412300075879103}}{{821132812500000000}}r + \frac{{5236655690345652768413}}{{1970718750000000000000}}.
\end{split}
\end{equation*}
Thus $T\left( r \right) >0$ if $r>2.57$. Hence we showed that for $r>2.57$, $\alpha \left(r\right)$ is a strictly decreasing function of $r$, satisfies the inequality $\alpha \left( r \right) < \frac{{\zeta \left( 3 \right)}}{{6\zeta \left( 3 \right) - 6}}$ and $\lim _{r \to +\infty } \alpha \left( r \right) = \frac{{13}}{{30}}$. Numerical evaluation shows that $\alpha \left( {2.57} \right) = 0.4709258826 \ldots  > \frac{{13}}{{30}} = 0.4333333333 \ldots$. From these it follows that $\frac{{13}}{{30}} < \alpha \left( r \right)$ and we conclude the statement of the lemma.
\end{proof}

\begin{proof}[Proof of Lemma~\ref{le2}.]
To prove the lemma we show that $T\left(r\right)$ defined by (\ref{eq6}) is positive if $0 < r \leq 0.3$. It is well known (see for example \cite{cerone}) that
\[
\sum\limits_{n \ge 1} {\frac{{2n}}{{\left( {n^2  + r^2 } \right)^\mu  }}}  = \sum\limits_{n \ge 0} {\left( { - 1} \right)^n 2\binom{\mu  + n - 1 }{ n} \zeta \left( {2\mu  + 2n - 1} \right)r^{2n} } ,\;\;\mu  > 0,\;\;0< r < 1.
\]
Here $\zeta$ denotes Riemann's Zeta function. And thus
\begin{equation*}
\begin{split}
S\left( r \right) = & \sum\limits_{n \ge 1} {\frac{{2n}}{{\left( {n^2  + r^2 } \right)^2 }}}  = \sum\limits_{n \ge 0} {\left( { - 1} \right)^n \left( {2n + 2} \right)\zeta \left( {2n + 3} \right)r^{2n} } ,\\
& \sum\limits_{n \ge 1} {\frac{{2n}}{{\left( {n^2  + r^2 } \right)^3 }}}  = \sum\limits_{n \ge 0} {\left( { - 1} \right)^n \left( {n + 1} \right)\left( {n + 2} \right)\zeta \left( {2n + 5} \right)r^{2n} } .
\end{split}
\end{equation*}
These series are alternating and for fixed $0<r\leq0.3$, the sequences
\begin{align*}
s_n & = \left( {2n + 2} \right)\zeta \left( {2n + 3} \right)r^{2n} ,\\
t_n & = \left( {n + 1} \right)\left( {n + 2} \right)\zeta \left( {2n + 5} \right)r^{2n} 
\end{align*}
are monotonically decreasing. By Leibniz's theorem we find
\begin{align*}
\sum\limits_{n \ge 1} {\frac{{2n}}{{\left( {n^2  + r^2 } \right)^2 }}}  & = 2\zeta \left( 3 \right) - 4\zeta \left( 5 \right)r^2  + 6\zeta \left( 7 \right)r^4  - \theta _r  \cdot 8\zeta \left( 9 \right)r^6 ,\\
\sum\limits_{n \ge 1} {\frac{{2n}}{{\left( {n^2  + r^2 } \right)^3 }}}  & = 2\zeta \left( 5 \right) - 6\zeta \left( 7 \right)r^2  + 12\zeta \left( 9 \right)r^4  - \rho _r  \cdot 20\zeta \left( 11 \right)r^6 ,
\end{align*}
where $0 < \theta _r ,\rho _r  < 1$ and $0<r\leq0.3$. By plugging these into (\ref{eq6}), we obtain
\begin{equation*}
\begin{split}
T\left( r \right) & > \left( {3r^4  + 3r^2  + \frac{1}{2}} \right)\left( {2\zeta \left( 3 \right) - 4\zeta \left( 5 \right)r^2  + 6\zeta \left( 7 \right)r^4  - 8\zeta \left( 9 \right)r^6 } \right)^2 \\
 & - \left( {6r^2  + 2} \right)\left( {2\zeta \left( 3 \right) - 4\zeta \left( 5 \right)r^2  + 6\zeta \left( 7 \right)r^4 } \right)\\
 & - \left( {2r^2  + \frac{1}{2}} \right)\left( {2\zeta \left( 5 \right) - 6\zeta \left( 7 \right)r^2  + 12\zeta \left( 9 \right)r^4 } \right) + 3 = Q\left(r\right).
\end{split}
\end{equation*}
The right-hand side is a polynomial in degree $16$. It can be shown that the polynomial\linebreak $\left( {r + \frac{{10}}{3}} \right)^{16} Q\left( {\frac{1}{{r + 10/3}}} \right)$ has only positive coefficients, thus $Q\left(r\right)$ is positive on the range $0<r \leq 0.3$, hence $\alpha$ is strictly decreasing there. This means that $\frac{{13}}{{30}} < \alpha \left( {0.3} \right) \le \alpha \left( r \right) < \alpha \left( {0^ +  } \right) = \frac{{\zeta \left( 3 \right)}}{{6\zeta \left( 3 \right) - 6}}$, which completes the proof of the lemma.
\end{proof}

\begin{proof}[Proof of Lemma~\ref{le3}.] As an application of the Euler--Maclaurin summation formula,\linebreak Lampret \cite{lampret} proved that
\[
S\left( r \right) = \sum\limits_{n = 1}^{m - 1} {\frac{{2n}}{{\left( {n^2  + r^2 } \right)^2 }}}  + \frac{1}{{m^2  + r^2 }} + \frac{m}{{\left( {m^2  + r^2 } \right)^2 }} + \frac{{3m^2  - r^2 }}{{6\left( {m^2  + r^2 } \right)^3 }} + \rho \left( {m,r} \right),
\]
where
\[
\left| {\rho \left( {m,r} \right)} \right| \le \frac{{5m^4  + 15m^2 r^2  + 6r^4 }}{{16\left( {m^2  + r^2 } \right)^5 }}.
\]
With $m=5$ we obtain
\[
S\left( r \right) > \sum\limits_{n = 1}^4 {\frac{{2n}}{{\left( {n^2  + r^2 } \right)^2 }}}  + \frac{1}{{25 + r^2 }} + \frac{5}{{\left( {25 + r^2 } \right)^2 }} + \frac{{75 - r^2 }}{{6\left( {25 + r^2 } \right)^3 }} - \frac{{3125 + 375r^2  + 6r^4 }}{{16\left( {25 + r^2 } \right)^5 }} = Z_1\left( r \right),
\]
and
\[
Z_1\left( r \right) - \frac{1}{{r^2  + 1/2}}>0,\;r>0.
\]
According to (\ref{eq2}), (\ref{eq3}) and (\ref{eq4}) we have
\[
\alpha \left( r \right) < \frac{{\left( {12r^4  + 2r^2  - 1} \right)Z_1\left( r \right) - 12r^2 }}{{12 - \left( {12r^2  + 6} \right)Z_1\left( r \right)}},\;r>0.
\]
The right-hand side is a rational function of $r$. Let
\[
\frac{d}{{dr}}\frac{{\left( {12r^4  + 2r^2  - 1} \right)Z_1\left( r \right) - 12r^2 }}{{12 - \left( {12r^2  + 6} \right)Z_1\left( r \right)}} = \frac{{\varphi_1 \left( r \right)}}{{\psi_1 \left( r \right)}},
\]
where $\varphi_1 \left( r \right)$ and $\psi_1 \left( r \right)$ are polynomials in $r$, $\psi_1 \left( r \right)>0$. $\varphi_1 \left( r \right)$ has degree 45 and \linebreak $\left( {r + \frac{1}{{2.57}}} \right)^{45} \varphi_1 \left( {\frac{1}{{r + 1/2.57}}} \right) < 0$ if $r>0$ since its every coefficient is negative. Hence
\[
\frac{{\left( {12r^4  + 2r^2  - 1} \right)Z_1\left( r \right) - 12r^2 }}{{12 - \left( {12r^2  + 6} \right)Z_1\left( r \right)}}
\]
is a strictly decreasing function of $r$ if $0 \leq r \leq 2.57$. Its value at $0.3$ is $0.9596637512\ldots < \frac{{\zeta \left( 3 \right)}}{{6\zeta \left( 3 \right) - 6}}$, thus $\alpha \left( r \right) < \frac{{\zeta \left( 3 \right)}}{{6\zeta \left( 3 \right) - 6}}$ if $0.3 < r \leq 2.57$. If $m=4$, Lampret's formula gives
\[
Z_2 \left( r \right)= \sum\limits_{n = 1}^3 {\frac{{2n}}{{\left( {n^2  + r^2 } \right)^2 }}}  + \frac{1}{{16 + r^2 }} + \frac{4}{{\left( {16 + r^2 } \right)^2 }} + \frac{{48 - r^2 }}{{6\left( {16 + r^2 } \right)^3 }} + \frac{{1280 + 240r^2  + 6r^4 }}{{16\left( {16 + r^2 } \right)^5 }} > S\left( r \right).
\]
Since $S\left( r \right) > \frac{1}{{r^2  + 1/2}}$, according to (\ref{eq2}), (\ref{eq3}) and (\ref{eq4}) we obtain
\[
\alpha \left( r \right) > \frac{{\left( {12r^4  + 2r^2  - 1} \right)Z_2 \left( r \right) - 12r^2 }}{{12 - \left( {12r^2  + 6} \right)Z_2 \left( r \right)}},\; r>0.
\]
Again, the right-hand side is a rational function of $r$, let
\[
\frac{d}{{dr}}\frac{{\left( {12r^4  + 2r^2  - 1} \right)Z_2 \left( r \right) - 12r^2 }}{{12 - \left( {12r^2  + 6} \right)Z_2 \left( r \right)}} = \frac{{\varphi _2 \left( r \right)}}{{\psi _2 \left( r \right)}},
\]
where $\varphi_2 \left( r \right)$ and $\psi_2 \left( r \right)$ are polynomials in $r$, $\psi_2 \left( r \right)>0$. $\varphi_2 \left( r \right)$ has only negative coefficients, hence
\[
\frac{{\left( {12r^4  + 2r^2  - 1} \right)Z_2\left( r \right) - 12r^2 }}{{12 - \left( {12r^2  + 6} \right)Z_2\left( r \right)}}
\]
is a strictly decreasing function of $r$ if $r>0$. Its value at $2.57$ is $0.4360975104\ldots > \frac{13}{30}$, thus $\alpha \left( r \right) > \frac{13}{30}$ if $0.3 < r \leq 2.57$ and this completes the proof.
\end{proof}

\section{Proof of Theorem~\ref{th1}}

From the lemmas we conclude that
\[
\alpha \left( { + \infty  } \right) = \frac{{13}}{{30}} < \alpha \left( r \right) < \frac{{\zeta \left( 3 \right)}}{{6\zeta \left( 3 \right) - 6}} = \alpha \left( 0^ +  \right),
\]
if $r>0$. By the definition of $\alpha \left( r \right)$ this is equivalent to
\[
\cfrac{1}{{r^2  + \cfrac{1}{2} - \cfrac{{4r^2  + 1}}{{12}}\left( {r^2  + \cfrac{{\zeta \left( 3 \right)}}{{6\zeta \left( 3 \right) - 6}}} \right)^{ - 1} }} < S\left( r \right) < \cfrac{1}{{r^2  + \cfrac{1}{2} - \cfrac{{4r^2  + 1}}{{12}}\left( {r^2  + \cfrac{{13}}{{30}}} \right)^{ - 1} }}.
\]
We have
\[
\mathop {\lim }\limits_{r \to 0^ +  } \cfrac{1}{{r^2  + \cfrac{1}{2} - \cfrac{{4r^2  + 1}}{{12}}\left( {r^2  + \cfrac{{\zeta \left( 3 \right)}}{{6\zeta \left( 3 \right) - 6}}} \right)^{ - 1} }} = 2\zeta \left( 3 \right) = S\left( {0^ +  } \right),
\]
thus the lower bound is sharp. From Lemma~\ref{th2}
\[
S\left( r \right) \sim \frac{1}{{r^2 }} - \frac{1}{{6r^4 }} - \frac{1}{{30r^6 }} - \frac{1}{{42r^8 }}- \frac{1}{{30r^{10} }} - \frac{5}{{66r^{12} }}-  \cdots ,
\]
as $r \rightarrow +\infty$. Plugging this into (\ref{eq1}) yields
\[
\alpha \left( r \right) \sim \frac{{13}}{{30}} + \frac{{104}}{{525r^2 }} + \frac{{592}}{{2625r^4 }} + \frac{{404032}}{{1010625r^6 }} +  \cdots ,
\]
as $r \rightarrow +\infty$. Hence we obtain a new asymptotic expansion to Mathieu's series:
\[
S\left( r \right) \sim \cfrac{1}{{r^2  + \cfrac{1}{2} - \cfrac{{4r^2  + 1}}{{12}}\left( {r^2  + \cfrac{{13}}{{30}} + \cfrac{{104}}{{525r^2 }} + \cfrac{{592}}{{2625r^4 }} + \cfrac{{404032}}{{1010625r^6 }} +  \cdots } \right)^{ - 1} }}
\]
holds as $r \rightarrow +\infty$. This shows that the upper bound is sharp too.

\end{document}